\documentclass[oneside,french]{amsart}
\usepackage[T1]{fontenc}
\usepackage[latin9]{inputenc}
\usepackage{amsthm}
\usepackage{amssymb}
\usepackage[dvips]{graphicx}
\makeatletter
\numberwithin{equation}{section}
\numberwithin{figure}{section}
\theoremstyle{plain}
\newtheorem{thm}{Thorme}
  \theoremstyle{plain}
  \newtheorem{prop}[thm]{Proposition}
  \theoremstyle{plain}
  \newtheorem{lem}[thm]{Lemme}
  \theoremstyle{plain}
  \newtheorem{rmk}[thm]{Remarque}
  \theoremstyle{plain}

\catcode`\=\active\def {\'e}
\catcode`\=\active\def {\`e}
\catcode`\=\active\def {\`a}
\catcode`\=\active\def {\^a}
\catcode`\Ë=\active\def Ë{\`A}
\catcode`\=\active\def {\"a}
\catcode`\=\active\def {\`u}
\catcode`\=\active\def {\^{e}}
\catcode`\=\active\def {\^{i}}
\catcode`\=\active\def {\"i}
\catcode`\=\active\def {\^{o}}
\catcode`\=\active\def {\"{o}}
\catcode`\=\active\def {\^{u}}
\catcode`\=\active\def {\"{u}}
\catcode`\=\active\def {\"{U}}
\catcode`\=\active\def {\c{c}}
\catcode`\=\active\def {\c{C}}
\catcode`\é=\active\def é{\c{C}}

\newcommand{\Rep}{\mathrm{Rep}}
\newcommand{\Isoc}{\mathrm{Isoc}}
\newcommand{\Hom}{\mathrm{Hom}}

\newcommand{\Aut}{\mathrm{Aut}}
\newcommand{\Int}{\mathrm{Int}}
\newcommand{\Spec}{\mathrm{Spec}}
\newcommand{\Gal}{\mathrm{Gal}}
\newcommand{\Res}{\mathrm{Res}}
\newcommand{\Min}{\mathrm{Min}}

\newcommand{\GL}{{\operatorname{GL}}}
\newcommand{\GU}{{\operatorname{GU}}}

\newcommand{\calD}{{\mathcal{D}}}

\def\Q{\mathbb{Q}}

\def\2vector#1#2{\left( \begin{smallmatrix} #1 \\ #2 \end{smallmatrix}
\right)}

\def\deb{ \begin{equation} }
\def\fin{ \end{equation} }

\usepackage[usenames]{color}
\definecolor{Indigo}{rgb}{0.2,0.1,0.7}
\definecolor{Violet}{rgb}{0.5,0.1,0.7}
\definecolor{White}{rgb}{1,1,1}
\definecolor{Green}{rgb}{0.1,0.9,0.2}

\makeatother

\usepackage{babel}
\addto\extrasfrench{\providecommand{\fg}{\ifdim\lastskip>\z@\unskip\fi~\frqq}}

\begin{document}

\title{Cristaux et immeubles}
\date{\today}

\author{Christophe Cornut et Marc-Hubert Nicole}

\maketitle

\begin{abstract}
Soit $G$ un groupe rductif connexe dfini sur $\Q_p$. L'ensemble des cristaux contenus dans un $G$-isocristal donn est envisag d'un point de vue immobilier comme un voisinage tubulaire d'un squelette caractris par une proprit de minimalit de nature mtrique.
\end{abstract}

\section{Introduction}

La catgorie des isocristaux sur un corps algbriquement clos de caractristique $p>0$ est semisimple, et ses objets simples sont classifis par leur pente $\lambda \in \Q$. Ce travail propose un point de vue immobilier systmatique sur l'ensemble des cristaux contenus dans un isocristal fix muni de structures supplmentaires.
 De tels ensembles interviennent notamment dans la description conjecturale des fibres spciales des varits de Shimura. La pertinence et le caractre structurant du point de vue immobilier apparaissent dj clairement dans les travaux de Vollaard et Wedhorn sur les cristaux supersinguliers pour les varits de groupe $\GU(n,1)$, voir \cite{VoWe11} et ses rfrences.

Dans la classification de Manin des $F$-cristaux  isomorphisme prs, certains cristaux dits spciaux jouent un rle de pivot. Dans chaque classe d'isognie, il n'y a qu'un nombre fini de classes d'isomorphisme de tels cristaux, mais ce nombre est gnralement suprieur  un. Parmi ceux-ci, Oort a identifi des cristaux qui sont uniquement dtermins par leur troncature  l'chelon $1$: les cristaux minimaux. Ils sont caractriss par cette proprit de rigidit et forment une unique classe d'isomorphisme, ce qui les dsigne comme un candidat plus naturel autour duquel amorcer la classification.

Dans le cadre du formalisme des $G$-isocristaux de Kottwitz et Rapoport-Richartz, nous proposons de placer au coeur de la description des ensembles de cristaux ci-dessus un lieu remarquable qui en constitue le squelette. Ce lieu est caractris par une proprit de minimalit que l'on emprunte  la thorie des espaces mtriques  courbure ngative, mais qui nous fut inspire par la notion homonyme de Oort. 

L'ensemble des cristaux contenus dans un $G$-isocristal fix est donc vu ici comme une sorte de voisinage tubulaire du squelette. Notre rsultat principal identifie ce dernier  l'immeuble de Bruhat-Tits du groupe des automorphismes du $G$-isocristal. Il nous reste toutefois  comprendre la structure combinatoire qui puisse dcrire les sections de ce voisinage tubulaire.

 



\section{Isocristaux\label{sec:Isocristaux}}

\subsection{~}

Soit $G$ un groupe algbrique rductif connexe sur $\mathbb{Q}_{p}$,
$k$ un corps algbriquement clos de caractristique $p$, $L$ le
corps des fractions de $W(k)$, et $\sigma$ le Frobenius de $L$.
On note $\Rep(G)$ la catgorie des reprsentations algbriques de
$G$, $\Isoc(k)$ la catgorie des isocristaux sur $k$ et $\calD$
le groupe diagonalisable sur $\mathbb{Q}_{p}$ dont le groupe des
caractres est $X^{\star}(\calD)=\mathbb{Q}$. Les rfrences
pour cette section sont \cite{Ko85,RaRi96,Ko97}.

\subsection{~}

A tout lment $b$ de $G(L)$, on associe les objets suivants:
\begin{enumerate}
\item le $G$-isocristal $\beta_{b}:\Rep(G)\rightarrow\Isoc(k)$,
\item son groupe d'automorphismes $J_{b}=\Aut^{\otimes}(\beta_{b})$, 
\item son morphisme de Newton $\nu_{b}:\calD_{L}\rightarrow G_{L}$,
et
\item le centralisateur $M_{b}=Z_{G_{L}}(\nu_{b})$ de $\nu_{b}$ dans $G_{L}$.
\end{enumerate}

\subsection{~}

Rappelons brivement la dfinition de ces objets. Tout d'abord, \[
\forall(V,\rho)\in\Rep(G):\qquad\beta_{b}(V,\rho)=\left(V\otimes L,\rho(b)\circ(\mathrm{Id}_{V}\otimes\sigma)\right).\]
Ensuite, $J_{b}$ est le schma en groupes sur $\mathbb{Q}_{p}$ tel
que pour toute $\mathbb{Q}_{p}$-algbre $R$, \[
J_{b}(R)=\left\{ g\in G(R\otimes L)\vert g\cdot b=b\cdot\sigma(g)\right\} .\]
Enfin, $\nu_{b}$ est caractris par la proprit suivante: pour
toute reprsentation $(V,\rho)$ de $G$ et tout lment $\lambda$
de $\mathbb{Q}=X^{\star}(\calD)$, la partie isocline de pente
$\lambda$ dans l'isocristal $\beta_{b}(V,\rho)$ est la composante
$\lambda$-isotypique pour l'action $\rho\circ\nu_{b}$ de $\calD_{L}$
sur $V\otimes L$. Il rsulte facilement de ces dfinitions (cf. \cite[4.4.3]{Ko85}
ou \cite[\S 1.7]{RaZi96}) que \begin{equation}
\Int(b)\circ\sigma(\nu_{b})=\nu_{b}\quad\mbox{dans }\Hom_{L}\left(\calD_{L},G_{L}\right).\label{eq:Fbnueqnu}\end{equation}

\subsection{~}

Le $\mathbb{Q}_{p}$-schma en groupes $J_{b}$ est une forme dfinie
sur $\mathbb{Q}_{p}$ du sous-groupe de Levi $M_{b}$ de $G_{L}$.
Plus prcisment, pour toute $L$-algbre $R$, le morphisme canonique
de $L$-algbres $R\otimes L\rightarrow R$, vu comme augmentation
de la $R$-algbre $R\otimes L$, induit une rtraction du morphisme
$G(R)\rightarrow G(R\otimes L)$, dont le compos \[
J_{b}(R)\hookrightarrow G(R\otimes L)\twoheadrightarrow G(R)\]
avec l'inclusion de $J_{b}(R)$ dans $G(R\otimes L)$ est un isomorphisme
de $J_{b}(R)$ sur le sous-groupe $M_{b}(R)$ de $G(R)=G_{L}(R)$,
cf. \cite[\S 3.3 et Appendix A]{Ko97}. On obtient ainsi un isomorphisme
de schmas en groupes sur $L$, que l'on note $\iota_{b}:J_{b,L}\rightarrow M_{b}$.

\subsection{~}

Identifions $\sigma(M_{b})\hookrightarrow\sigma(G_{L})=G_{L}$ au
centralisateur de $\sigma(\nu_{b})$ dans $G_{L}$. L'quation~\ref{eq:Fbnueqnu}
montre alors que l'automorphisme $\Int(b)$ de $G_{L}$ induit un
isomorphisme $\sigma(M_{b})\rightarrow M_{b}$ de schmas en groupes
sur $L$, qui n'est autre que la {}``donne de descente'' sur $M_{b}$
induite par $\iota_{b}$, i.e. le diagramme suivant est commutatif:\begin{equation}
\begin{array}{ccccccc}
J_{b,L} & \stackrel{\iota_{b}}{\longrightarrow} & M_{b} & \hookrightarrow & G_{L} & \rightarrow & \Spec(L)\\
\uparrow &  & \uparrow &  & \uparrow &  & \uparrow\sigma\\
J_{b,L} & \stackrel{\sigma(\iota_{b})}{\longrightarrow} & \sigma(M_{b}) & \hookrightarrow & G_{L} & \rightarrow & \Spec(L)\\
\parallel &  & \downarrow &  & \downarrow &  & \parallel\\
J_{b,L} & \stackrel{\iota_{b}}{\longrightarrow} & M_{b} & \hookrightarrow & G_{L} & \rightarrow & \Spec(L)\end{array}\label{eq:EquivIota}\end{equation}
Dans ce diagramme, tous les carrs du haut sont cartsiens et ne servent
qu' dfinir la seconde ligne. La commutativit des carrs du bas
(o les flches verticales du carr central sont induites par $\Int(b)$)
est essentiellement une tautologie.


\subsection{~}

Puisque $\calD$ est commutatif, le morphisme $\nu_{b}:\calD_{L}\rightarrow G_{L}$
se factorise en un morphisme central $\calD_{L}\rightarrow M_{b}$
que l'on note encore $\nu_{b}$. On dduit de (\ref{eq:Fbnueqnu})
et (\ref{eq:EquivIota}) que le morphisme $\iota_{b}^{-1}\circ\nu_{b}:\calD_{L}\rightarrow J_{b,L}$
est $\sigma$-invariant. Il provient donc d'un morphisme central $\nu_{b}^{J}:\calD\rightarrow J_{b}$
de schmas en groupes sur $\mathbb{Q}_{p}$.

\subsection{~}

Soit $g\in G(L)$ et $b'=gb\sigma(g)^{-1}$. Alors $g$ induit un
isomorphisme $\beta_{b}\rightarrow\beta_{b'}$, donc aussi un isomorphisme
$J(g):J_{b}\rightarrow J_{b'}$. Pour toute $\mathbb{Q}_{p}$-algbre
$R$, l'isomorphisme $J(g):J_{b}(R)\rightarrow J_{b'}(R)$ est induit
par l'automorphisme $\Int(g)$ de $G(R\otimes L)$. Enfin $\Int(g)\circ\nu_{b}=\nu_{b'}$,
$\Int(g)(M_{b})=M_{b'}$ et $\Int(g)\circ\iota_{b}=\iota_{b'}\circ J(g)$.

\subsection{~\label{sub:decence}}

On dit d'un lment $b$ de $G(L)$ qu'il est dcent si et seulement
s'il existe un entier $s\geq1$ tel que $s\nu_{b}:\calD_{L}\rightarrow G_{L}$
se factorise  travers un cocaractre $(s\nu_{b}):\mathbb{G}_{m,L}\rightarrow G_{L}$
pour lequel l'quation suivante, dite de dcence, est vrifie: \begin{equation}
(b,\sigma)^{s}=\left((s\nu_{b})(p),\sigma^{s}\right)\quad\mbox{dans }G(L)\rtimes\left\langle \sigma\right\rangle .\label{eq:Decence}\end{equation}
On vrifie qu'alors $b$, $\nu_{b}$, $M_{b}$ et $\iota_{b}:J_{b,L}\rightarrow M_{b}$
sont dfinis sur la sous-extension $\mathbb{Q}_{p^{s}}$ de $L/\mathbb{Q}_{p}$,
cf. \cite[ Cor. 1.9]{RaZi96}, et l'isomorphisme $\sigma(M_{b})\rightarrow M_{b}$
induit par $\Int(b)$ est alors une authentique donne de descente
sur $M_{b}$ relativement  l'extension finie et non-ramifie $\mathbb{Q}_{p^{s}}/\mathbb{Q}_{p}$,
cf. \cite[ Cor. 1.14]{RaZi96}. La description explicite de $b\mapsto\nu_{b}$
dans \cite[\S 4.3]{Ko85} montre que tout lment $b$ de $G(L)$
est $\sigma$-conjugu  un lment dcent.

\subsection{~}

Lorsque $G$ est quasi-dploy sur $\mathbb{Q}_{p}$, tout lment
de $G(L)$ est $\sigma$-conjugu  un lment $b$ tel que $\nu_{b}:\calD_{L}\rightarrow G_{L}$
est dfini sur $\mathbb{Q}_{p}$ d'aprs la preuve de \cite[6.2]{Ko85},
qui renvoie  \cite[1.1.3]{Ko84}. Alors $M_{b}$ est galement dfini
sur $\mathbb{Q}_{p}$, $b$ est un lment basique et $G$-rgulier
de $M_{b}(L)$, et $J_{b}$ est une forme intrieure de $M_{b}$
-- cf \cite[\S 6]{Ko85}.

\section{Immeubles}

On souhaite maintenant dcrire les immeubles de Bruhat-Tits des groupes
rductifs de la section prcdente. Bien que les corps de bases ($\mathbb{Q}_{p}$,
$\mathbb{Q}_{p^{s}}$ ou $L$) y soient tous complets et absolument
non-ramifis, les questions de descentes sur les immeubles nous contraindront
ci-dessous  travailler aussi sur la  clture algbrique $\mathbb{Q}_{p}^{nr}$
de $\mathbb{Q}_{p}$ dans $L$. On se place donc plus gnralement
sur un corps $K$ de caractristique zro, muni d'une valuation $\omega:K^{\times}\rightarrow\mathbb{R}$
que l'on suppose discrte et non-triviale, et telle que l'anneau de
valuation $\{x\in K\vert\omega(x)\geq0\}$ soit henslien avec un
corps rsiduel parfait -- ce sont les corps quasi-locaux de caractristique
zro considrs dans \cite{La00}.

\subsection{~}

On note $\mathcal{I}(H,K)$ l'immeuble de Bruhat-Tits d'un groupe
rductif connexe $H$ sur $K$. C'est l'immeuble not $\mathcal{B}(H,K)$
dans \cite[\S 2.1]{Ti79} (o $K$ est complet), l'immeuble centr
not $\mathcal{I}_{K}(H)$ dans \cite[2.1.15]{Ro77}, ou enfin l'immeuble
tendu not $\mathcal{BT}^{e}(H,K)$ dans \cite[1.3.2]{La00}. En
particulier, \begin{equation}
\mathcal{I}(H,K)=\mathcal{I}'(H,K)\times V(H,K)\quad\mbox{o}\quad V(H,K)=X_{\star}^{K}\left(Z(H)\right)\otimes\mathbb{R}.\label{eq:DecImm}\end{equation}
L'immeuble $\mathcal{I}'(H,K)$ est not $\mathcal{I}'_{K}(H)$ dans
\cite{Ro77} et $\mathcal{BT}(H,K)$ dans \cite{La00}; il est canoniquement
isomorphe  l'immeuble $\mathcal{I}(H',K)$ du groupe driv $H'$
de $H$. Par ailleurs, on a not $X_{\star}^{K}(Z(H))=X_{\star}(Z^{K}(H))$
le groupe des cocaractres du plus gros sous-tore dploy $Z^{K}(H)$
du centre $Z(H)$ de $H$. On note $x'$ et $x^{v}$ les composantes
d'un point $x=(x',x^{v})$ de $\mathcal{I}(H,K)=\mathcal{I}'(H,K)\times V(H,K)$.

\subsection{~}

L'immeuble $\mathcal{I}(H,K)$ est muni d'une action  gauche de $H(K)$,
d'une structure poly-simpliciale invariante sous $H(K)$ (qui donne
une partition en facettes), et il est recouvert par des espaces affines
(les appartements) qui sont des runions de facettes; il y a une bijection
$H(K)$-quivariante $S\mapsto\mathcal{A}(S)$ entre l'ensemble des
tores dploys maximaux $S$ de $H_{K}$ et l'ensemble des appartements
de $\mathcal{I}(H,K)$, l'espace vectoriel sous-jacent  $\mathcal{A}(S)$
est gal  $X_{\star}(S)\otimes\mathbb{R}$, l'action du normalisateur
$\mathcal{N}(S)$ de $S$ dans $H(K)$ sur $\mathcal{A}(S)$ est affine,
et la partie vectorielle de cette action est l'action par conjugaison
de $\mathcal{N}(S)$ sur $X_{\star}(S)$. En outre, on peut munir
$\mathcal{I}(H,K)$ d'une mtrique $H(K)$-invariante, dont la restriction
 chaque appartement est euclidienne, et qui fait de $\mathcal{I}(H,K)$
un espace mtrique complet, godsique, et CAT(0).

\subsection{~\label{sub:CompStructImmEt}}

Ces structures sont compatibles avec la dcomposition (\ref{eq:DecImm}).
En particulier, les applications $F'\mapsto F'\times V(H,K)$ et \emph{$\mathcal{A}'\mapsto\mathcal{A}'\times V(H,K)$}
sont des bijections $H(K)$-quivariantes de l'ensemble des facettes
(resp. appartements) de $\mathcal{I}'(H,K)$ sur ceux de $\mathcal{I}(H,K)$.
L'action de $h\in H(K)$ sur $x=(x',x^{v})\in\mathcal{I}(H,K)$ est
\[
h\cdot x=\left(h\cdot x',x^{v}+v_{H}(h)\right)=\left(h\cdot x',x^{v}\right)+v_{H}(h)\]
pour un vecteur $v_{H}(h)\in V(H,K)$, et le morphisme \[
v_{H}:H(K)\rightarrow V(H,K)\]
est caractris par la formule suivante: pour tout caractre $\chi:H\rightarrow\mathbb{G}_{m,K}$,
\[
\left\langle \nu_{H}(h),\chi\vert Z^{K}(H)\right\rangle _{\mathbb{R}}=-\omega(\chi(h))\]
o $\left\langle -,-\right\rangle _{\mathbb{R}}$ est l'extension
$\mathbb{R}$-linaire de l'accouplement usuel \[
\left\langle -,-\right\rangle :X_{\star}(Z^{K}(H))\times X^{\star}(Z^{K}(H))\rightarrow\mathbb{Z}.\]
Contrairement aux autres structures, la mtrique euclidienne $d$
sur $\mathcal{I}(H,K)$ n'est pas tout  fait canoniquement associe
 $(H,K)$. Elle se dcompose en \[
d(x,y)=\sqrt{d'(x',y')^{2}+\left(x^{v},y^{v}\right)^{2}}\]
o $d'$ est une mtrique $H(K)$-invariante sur $\mathcal{I}'(H,K)$
qui est uniquement dtermine  multiplication par un scalaire prs,
et que l'on peut mme normaliser de diverses manires, mais o $\left(-,-\right)$
est un produit scalaire arbitraire sur $V(H,K)$.

\subsection{~}

La fonctorialit en $H$ et/ou $K$ de l'immeuble $\mathcal{I}(H,K)$
est une question dlicate%
\footnote{On peut esprer que la construction dans \cite{ReThWe10} d'une ralisation
intrinsque de l'immeuble $\mathcal{I}(H,K)$ dans l'espace de Berkovich
$H^{an}$ des normes multiplicatives sur l'algbre de Hopf $K[H]$
de $H$, permettra dans un avenir proche de clarifier ces questions
de fonctorialit. %
} (cf. \cite{Ro77,PrYu02,La00}), mais la dmonstration des cas dont
nous aurons besoin ci-dessous remonte aux travaux de Bruhat et Tits.

\subsubsection{~\label{sub:FonctAut}}

Tout d'abord, le groupe $(\Aut H)(K)$ agit sur $\mathcal{I}(H,K)$
par {}``transport de structure'' - dixit \cite[\S 2.5]{Ti79}. On
en dduit une seconde action de $H(K)$ sur $\mathcal{I}(H,K)$ --
par automorphisme intrieur. Cette seconde action concide avec l'action
de $H(K)$ sur le facteur $\mathcal{I}'(H,K)$, sur l'ensemble des
facettes, et sur l'ensemble des appartements, mais elle est triviale
sur le facteur euclidien $V(H,K)$. On a donc\[
\Int(h)(x)=(h\cdot x',x^{v})=h\cdot(x',x^{v})-v_{H}(h)=h\cdot x-v_{H}(h)\]
pour tout $x=(x',x^{v})\in\mathcal{I}(H,K)$ et tout $h\in H(K)$.

\subsubsection{~\label{sub:FonctGal}}

Soit ensuite $(L,\omega)$ une extension (quasi-locale) de $(K,\omega)$,
que l'on suppose Galoisienne. Alors $\Gal(L/K)$ agit sur l'immeuble
$\mathcal{I}(H,L)$, et cette action est compatible avec l'action
de $H(L)$, i.e. \[
\gamma(h\cdot x)=\gamma(h)\cdot\gamma(x)\qquad\left(\gamma\in\Gal(L/K),h\in H(L),x\in\mathcal{I}(H,L)\right).\]
Elle est aussi compatible avec la dcomposition en facettes, avec
les appartements, avec la bijection $S\mapsto\mathcal{A}(S)$ entre
tores dploys maximaux de $H_{L}$ et appartements de $\mathcal{I}(H,L)$,
avec la dcomposition $\mathcal{I}(H,L)=\mathcal{I}(H',L)\times V(H,L)$,
et avec l'action naturelle de $\Gal(L/K)$ sur le facteur vectoriel
$V(H,L)=X_{\star}^{L}(Z(H_{L}))\otimes\mathbb{R}$. Cette action est
enfin compatible avec la mtrique euclidenne sur $\mathcal{I}(H,L)$,
pour peu que celle-ci soit induite par un produit scalaire $\Gal(L/K)$-invariant
sur $V(H,L)$. De plus, il existe un plongement canonique de $\mathcal{I}(H,K)$
dans $\mathcal{I}(H,L)$ qui est $G(K)$-quivariant, $\Gal(L/K)$-invariant,
compatible avec les dcompositions~\ref{eq:DecImm} et avec le plongement
naturel de $V(H,K)$ dans $V(H,L)$, et isomtrique pour un choix
convenable des mtriques euclidiennes. Si $L/K$ est non-ramifie,
ce plongement identifie $\mathcal{I}(H,K)$ au sous-espace $\mathcal{I}(H,L)^{\Gal(L/K)}$.
Voir \cite[\S 2.6]{Ti79}, \cite[\S 2.1]{La00}.

\subsubsection{~\label{sub:FonctRes}}

Si de plus $L/K$ est finie, alors $\mathcal{I}(H,L)$ s'identifie
canoniquement  $\mathcal{I}(G,K)$, o $G=\Res_{K}^{L}(H)$ est la
restriction  la Weil de $H$ -- voir \cite[2.1]{Ti79}. Dans cette
identification, l'action de $\Gal(L/K)$ sur $\mathcal{I}(H,L)$ considre
en (\ref{sub:FonctGal}) correspond  l'action de $\Gal(L/K)\subset\Aut(G)(K)$
sur $\mathcal{I}(G,K)$ considre en (\ref{sub:FonctAut}).

\subsubsection{~\label{sub:FonctLevi}}

Soit $M$ un sous-groupe de Levi de $H$. Les tores dploys maximaux
de $M$ sont les tores dploys maximaux de $H$ qui sont contenus
dans $M$. On note \[
\mathcal{I}'_{M}(H,K)\subset\mathcal{I}'(H,K)\quad\mbox{et}\quad\mathcal{I}_{M}(H,K)\subset\mathcal{I}(H,K)\]
la runion des appartements qui correspondent  ces tores, de sorte
que \[
\mathcal{I}_{M}(H,K)=\mathcal{I}'_{M}(H,K)\times V(H,K).\]
D'autre part, l'isognie $Z^{K}(H)\times H'\cap Z^{K}(M)\rightarrow Z^{K}(M)$
induit une dcomposition\[
V(M,K)=V'(M,K)\times V(H,K)\]
o l'on note \[
V'(M,K)=X_{\star}\left(H'\cap Z^{K}(M)\right)\otimes\mathbb{R}.\]
Pour tout tore dploy maximal $S$ de $M$, ce sous-espace vectoriel
de $X_{\star}(H'\cap S)\otimes\mathbb{R}$ agit sur l'appartement
$\mathcal{A}'(S)$ de $\mathcal{I}'(H,K)$ par translation, et ces
actions se recollent en une action $(v,x)\mapsto x+v$ de $V'(M,K)$
sur $\mathcal{I}'_{M}(H,K)$. Il existe une injection \[
\mathcal{I}'(M,K)\hookrightarrow\mathcal{I}'_{M}(H,K)\]
qui est $M(K)$-quivariante, compatible avec l'indexation $S\mapsto\mathcal{A}'(S)$
des appartements, bien dtermine  translation prs par un lment
de $V'(M,K)$, et qui induit une bijection $M(K)$-quivariante de
$\mathcal{I}(M,K)$ sur $\mathcal{I}_{M}(H,K)$,  savoir \begin{eqnarray*}
\mathcal{I}'(M,K)\times V'(M,K)\times V(H,K) & \rightarrow & \mathcal{I}'_{M}(H,K)\times V(H,K)\\
(x,v,w) & \mapsto & (x+v,w)\end{eqnarray*}
Rfrence: \cite[2.1.4 et 2.1.5]{La00}.

\subsubsection{~\label{sub:FonctCompl}}

Si $\widehat{K}$ est le complt de $K$, alors $\mathcal{I}(H,K)=\mathcal{I}(H,\widehat{K})$.
Rfrence: \cite[2.1.3]{La00}.

\subsection{Isomtries\label{sub:Isom=0000E9tries}}

Pour une isomtrie $\mathcal{F}$ d'un espace mtrique $\mathcal{I}$,
on note \begin{eqnarray*}
\min\left(\mathcal{F}\right) & = & \inf\left\{ \mathrm{dist}(x,\mathcal{F}x):x\in\mathcal{I}\right\} \\
\Min\left(\mathcal{F}\right) & = & \left\{ x\in\mathcal{I}:\mathrm{dist}(x,\mathcal{F}x)=\min(\mathcal{F})\right\} \end{eqnarray*}
On dit que $\mathcal{F}$ est semi-simple si et seulement si $\Min(\mathcal{F})\neq\emptyset$,
et on dit alors que $\mathcal{F}$ est elliptique ou hyperbolique
selon que $\min\left(\mathcal{F}\right)=0$ ou $\neq0$. On renvoie
 \cite[II.1]{BrHa99} pour la dfinition des espaces mtriques CAT(0)
- tous les immeubles considrs ci-dessus sont de tels espaces d'aprs
\emph{\cite[Exercice II.1.9.c]{BrHa99}} et \cite[Lemme 3.2.1]{BrTi72}.
\begin{prop}
\label{pro:Isom}Soit $\mathcal{F}$ une isomtrie d'un espace mtrique
$\mathcal{I}$ qui est CAT(0).
\begin{enumerate}
\item Le sous-espace $\Min\left(\mathcal{F}\right)$ de $\mathcal{I}$ est
convexe et ferm.
\item Pour toute isomtrie $\alpha$ de $\mathcal{I}$, \[
\min\left(\alpha\mathcal{F}\alpha^{-1}\right)=\min\left(\mathcal{F}\right)\quad\mbox{et}\quad\Min\left(\alpha\mathcal{F}\alpha^{-1}\right))=\alpha\left(\Min(\mathcal{F})\right).\]

\item Soit $\mathcal{C}\subset\mathcal{I}$ une partie non-vide, complte,
convexe et $\mathcal{F}$-invariante. Alors \[
\min\left(\mathcal{F}\right)=\min\left(\mathcal{F}\vert\mathcal{C}\right)\quad\mbox{et}\quad\Min\left(\mathcal{F}\vert\mathcal{C}\right)=\Min\left(\mathcal{F}\right)\cap\mathcal{C}.\]
De plus: $\mathcal{F}$ est semi-simple si et seulement si $\mathcal{F}\vert\mathcal{C}$
est semi-simple.
\item Pour tout $s>0$, $\mathcal{F}$ est semi-simple si et seulement si
$\mathcal{F}^{s}$ l'est, et alors \[
\min\left(\mathcal{F}^{s}\right)=s\min\left(\mathcal{F}\right)\quad\mbox{et}\quad\Min\left(\mathcal{F}\right)\subset\Min\left(\mathcal{F}^{s}\right).\]
 
\item L'isomtrie $\mathcal{F}$ est hyperbolique si et seulement si il
existe une godsique $c:\mathbb{R}\rightarrow\mathcal{I}$ et un
lment $a>0$ de $\mathbb{R}$ tel que $\mathcal{F}c(t)=c(t+a)$
pour tout $t\in\mathbb{R}$. Alors $a=\min(\mathcal{F})$. On dit
que $c(\mathbb{R})$ est un axe de $\mathcal{F}$. 
\item Si $\mathcal{F}$ est hyperbolique, les axes de $\mathcal{F}$ sont
tous parallles, leur runion est gale  $\Min\left(\mathcal{F}\right)$,
et $\Min\left(\mathcal{F}\right)$ est isomtrique  un produit $\Min'\left(\mathcal{F}\right)\times\mathbb{R}$
sur lequel $\mathcal{F}$ agit par $(x,t)\mapsto(x,t+\min(\mathcal{F}))$.
Toute isomtrie $\alpha$ de $\mathcal{I}$ qui commute avec $\mathcal{F}$
prserve $\Min\left(\mathcal{F}\right)$ et s'y dcompose en $\alpha\vert\Min\left(\mathcal{F}\right)=(\alpha',\alpha^{v})$
o $\alpha'$ est une isomtrie de $\Min'\left(\mathcal{F}\right)$
et $\alpha^{v}$ une translation de $\mathbb{R}$. Enfin $\alpha$
est semi-simple si et seulement si $\alpha'$ l'est, et $\Min\left(\alpha\right)\cap\Min\left(\mathcal{F}\right)=\Min\left(\alpha'\right)\times\mathbb{R}$. 
\end{enumerate}
\end{prop}
\begin{proof}
Voir \cite[II.6.2, 6.7, 6.8 et 6.9]{BrHa99}.
\end{proof}

\subsection{Le thorme principal}

On reprend les hypothses et notations de la section~\ref{sec:Isocristaux}:
$G$ est un groupe rductif sur $\mathbb{Q}_{p}$ et $L$ est le complt
de l'extension non-ramifie maximale $\mathbb{Q}_{p}^{nr}=\cup_{s\geq0}\mathbb{Q}_{p^{s}}$
de $\mathbb{Q}_{p}$. 

On fixe une mtrique euclidienne invariante sous $G(L)\rtimes\left\langle \sigma\right\rangle $
sur l'immeuble $\mathcal{I}(G,L)$, comme expliqu ci-dessus, notamment
dans la section~\ref{sub:FonctGal}. Pour tout lment $b$ de $G(L)$,
on note $\mathcal{F}_{b}$ l'isomtrie de $\mathcal{I}(G,L)$ induite
par l'lment $(b,\sigma)$ de $G(L)\rtimes\left\langle \sigma\right\rangle $.
On dfinit galement une fonction $G(L)\rtimes\left\langle \sigma\right\rangle $-invariante
\[
\min:\Hom\left(\calD_{L},G_{L}\right)\rightarrow\mathbb{R}_{+}\]
de la manire suivante. Soit $\nu\in\Hom\left(\calD_{L},G_{L}\right)$
et $S$ un tore dploy maximal de $G_{L}$ tel que $\nu\in\Hom\left(\calD_{L},S\right)=X_{\star}(S)\otimes\mathbb{Q}$.
On note $\nu_{S}\in X_{\star}(S)\otimes\mathbb{R}$ la translation
correspondante de l'appartement $\mathcal{A}(S)$ de $\mathcal{I}(G,L)$.
Si $\theta\in G(L)\rtimes\left\langle \sigma\right\rangle $, alors
$\theta(\nu)$ se factorise  travers $\theta(S)$, et $\theta$ induit
une isomtrie de\emph{ $\mathcal{I}(G,L)$ }qui envoie $\mathcal{A}(S)$
sur $\mathcal{A}(\theta(S))$ en entrelaant les translations $\nu_{S}$
et $\theta(\nu)_{\theta(S)}$. La fonction $\min(\nu)=\min(\nu_{S})$
est donc bien dfinie (elle ne dpend pas du choix de $S$) et $G(L)\rtimes\left\langle \sigma\right\rangle $-invariante.
\begin{thm}
Avec les notations ci-dessus, pour tout $b\in G(L)$, 
\begin{enumerate}
\item L'isomtrie $\mathcal{F}_{b}$ de $\mathcal{I}(G,L)$ est semi-simple
et $\min(\mathcal{F}_{b})=\min(\nu_{b}).$
\item L'action de $J_{b}(\mathbb{Q}_{p})=\left\{ g\in G(L)\vert gb=b\sigma(g)\right\} $
sur $\mathcal{I}(G,L)$ stabilise $\Min(\mathcal{F}_{b})$.
\item Il existe une bijection $J_{b}(\mathbb{Q}_{p})$-quivariante \[
\mathcal{I}(J_{b},\mathbb{Q}_{p})\stackrel{\simeq}{\longrightarrow}\Min(\mathcal{F}_{b})\subset\mathcal{I}(G,L)\]
qui est bien dfinie  translation prs par un lment de $V(J_{b},\mathbb{Q}_{p})$.
\end{enumerate}
\end{thm}
\begin{proof}
Le deuxime point du thorme est vident, et c'est un cas particulier
du deuxime point de la proposition~\ref{pro:Isom}: puisque $J_{b}(\mathbb{Q}_{p})$
est le commutant de $(b,\sigma)$ dans $G(L)$, l'action de $J_{b}(\mathbb{Q}_{p})$
sur $\mathcal{I}(G,L)$ commute  $\mathcal{F}_{b}$ et stabilise
donc $\Min(\mathcal{F}_{b})$. D'autre part, si $b'=gb\sigma(g)^{-1}$
pour un lment $g$ de $G(L)$, alors \[
\mathcal{F}_{b'}=g\mathcal{F}_{b}g^{-1}\quad\mbox{et}\quad\nu_{b'}=\Int(g)\circ\nu_{b}\]
donc $\Min(\mathcal{F}_{b'})=g\left(\Min(\mathcal{F}_{b})\right)$,
$\min(\mathcal{F}_{b'})=\min(\mathcal{F}_{b})$ et $\min(\nu_{b'})=\min(\nu_{b})$.
Enfin, l'isomorphisme $J(g):J_{b}\rightarrow J_{b'}$ induit des isomorphismes
$J(g)$-quivariants \[
\mathcal{I}'(J_{b},\mathbb{Q}_{p})\rightarrow\mathcal{I}'(J_{b'},\mathbb{Q}_{p})\quad\mbox{et}\quad V(J_{b},\mathbb{Q}_{p})\rightarrow V(J_{b'},\mathbb{Q}_{p}).\]
Quitte  remplacer $b$ par $b'$, on peut donc d'aprs~\ref{sub:decence}
supposer qu'il existe un entier $s>0$ tel que $s\nu_{b}:\calD_{L}\rightarrow G_{L}$
se factorise en un cocaractre $(s\nu_{b}):\mathbb{G}_{m,L}\rightarrow G_{L}$
avec \[
(b,\sigma)^{s}=\left(s\nu_{b}(p),\sigma^{s}\right)\quad\mbox{dans}\quad G(L)\rtimes\left\langle \sigma\right\rangle .\]
On rappelle qu'alors $b$, $\nu_{b}$, $M_{b}$ et $\iota_{b}$ sont
dfinis sur $\mathbb{Q}_{p^{s}}$. On choisit $s$ suffisamment grand
pour qu'aussi $V(G,L)=V(G,\mathbb{Q}_{p^{s}})$ et $V(M_{b},L)=V(M_{b},\mathbb{Q}_{p^{s}})$.
On pose \[
\mathcal{F}=\mathcal{F}_{b},\quad\nu=\nu_{b},\quad M=M_{b},\quad J=J_{b}\quad\mbox{et}\quad\iota=\iota_{b}:J_{\mathbb{Q}_{p^{s}}}\rightarrow M.\]
Le lemme suivant dmontre et prcise la premire assertion du thorme.
\begin{lem}
\label{lem:CalcMin}Avec les notations ci-dessus, \[
\begin{array}{rcccl}
\Min\left(\mathcal{F}^{s}\right) & = & \left\{ x\in\mathcal{I}_{M}(G,L):\mathcal{F}^{s}x=x+s\nu\right\}  & = & \mathcal{I}_{M}\left(G,\mathbb{Q}_{p^{s}}\right)\\
\mbox{et}\qquad\Min\left(\mathcal{F}\right) & = & \left\{ x\in\mathcal{I}_{M}(G,\mathbb{Q}_{p^{s}}):\mathcal{F}x=x+\nu\right\}  & \neq & \emptyset\end{array}\]
\end{lem}
\begin{proof}
Puisque $\Int(b)\circ\sigma(\nu)=\nu$, $\Int(b)\circ\sigma(M)=M$
et $\mathcal{F}$ stabilise $\mathcal{C}=\mathcal{I}_{M}\left(G,\mathbb{Q}_{p^{s}}\right)$.
L'quation de dcence montre que $\mathcal{F}^{s}$ agit sur $\mathcal{C}$
comme la translation de vecteur $v_{M}(s\nu(p))=s\nu\in V(M,\mathbb{Q}_{p^{s}})$,
donc $\mathcal{F}^{s}\vert\mathcal{C}$ est semi-simple et $\min\left(\mathcal{F}^{s}\vert\mathcal{C}\right)=s\min(\nu)$.
D'aprs $(3)$ et $(4)$ de la proposition~\ref{pro:Isom}, $\mathcal{F}$
est semi-simple et $\min(\mathcal{F})=\min(\nu)$. D'aprs $(5)$
et $(6)$ (si $\nu\neq0$), les axes de $\mathcal{F}$ (resp. $\mathcal{F}^{s}$)
sont des droites parallles de direction $\nu$, et leur runion est
gale  $\Min(\mathcal{F})$ (resp. $\Min(\mathcal{F}^{s})$). La
runion de toutes les droites parallles  ces axes n'est autre que
le sous-espace $\mathcal{I}_{M}(G,L)$ de $\mathcal{I}(G,L)$ -- voir
\cite[\S 3.9]{Ro01}. Dans tous les cas, on a donc \[
\Min(\mathcal{F})\subset\Min(\mathcal{F}^{s})\subset\mathcal{I}_{M}(G,L)\]
 avec plus prcisment \begin{eqnarray*}
\Min(\mathcal{F}) & = & \left\{ x\in\mathcal{I}_{M}(G,L):\mathcal{F}x=x+\nu\right\} ,\\
\Min(\mathcal{F}^{s}) & = & \left\{ x\in\mathcal{I}_{M}(G,L):\mathcal{F}^{s}x=x+s\nu\right\} .\end{eqnarray*}
Choisissons comme en~\ref{sub:FonctLevi} une isomtrie $M(L)$-quivariante\[
\mathcal{I}(M,L)\stackrel{\simeq}{\longrightarrow}\mathcal{I}_{M}(G,L)\]
dont on peut aussi supposer, d'aprs \cite[Proposition 2.1.5]{La00},
qu'elle commute  l'action de $\sigma^{s}$. L'quation de dcence
montre que l'action de $\mathcal{F}^{s}$ sur $\mathcal{I}_{M}(G,L)$
correspond  l'action de $(s\nu(p),\sigma^{s})$ sur $\mathcal{I}(M,L)=\mathcal{I}'(M,L)\times V(M,L)$,
i.e.  \[
(x',x^{v})\mapsto\left(\sigma^{s}x',\sigma^{s}x^{v}+s\nu\right)=\left(\sigma^{s}x',x^{v}\right)+s\nu.\]
On en dduit en utilisant \ref{sub:FonctGal} (et \ref{sub:FonctCompl})
que \[
\Min\left(\mathcal{F}^{s}\right)=\mathcal{C}=\mathcal{I}_{M}(G,\mathbb{Q}_{p^{s}})\quad\mbox{et}\quad\Min\left(\mathcal{F}\right)=\left\{ x\in\mathcal{C}:\mathcal{F}x=x+\nu\right\} \]
ce qui achve la dmonstration du lemme.
\end{proof}
Plongeons maintenant l'immeuble $\mathcal{I}(J,\mathbb{Q}_{p})$ dans
$\mathcal{I}_{M}(G,\mathbb{Q}_{p^{s}})$. Pour cela, commenons par
choisir comme en \ref{sub:FonctLevi} une isomtrie $M(\mathbb{Q}_{p^{s}})$-quivariante
\[
\theta:\mathcal{I}(M,\mathbb{Q}_{p^{s}})\rightarrow\mathcal{I}_{M}(G,\mathbb{Q}_{p^{s}}).\]
Soit $\mathcal{G}_{\theta}$ l'isomtrie de $\mathcal{I}_{M}(G,\mathbb{Q}_{p^{s}})$
dfinie par la commutativit du diagramme \[
\begin{array}{ccccc}
\mathcal{I}(J,\mathbb{Q}_{p^{s}}) & \stackrel{\iota}{\longrightarrow} & \mathcal{I}(M,\mathbb{Q}_{p^{s}}) & \stackrel{\theta}{\longrightarrow} & \mathcal{I}_{M}(G,\mathbb{Q}_{p^{s}})\\
\sigma\downarrow &  &  &  & \downarrow\mathcal{G}_{\theta}\\
\mathcal{I}(J,\mathbb{Q}_{p^{s}}) & \stackrel{\iota}{\longrightarrow} & \mathcal{I}(M,\mathbb{Q}_{p^{s}}) & \stackrel{\theta}{\longrightarrow} & \mathcal{I}_{M}(G,\mathbb{Q}_{p^{s}})\end{array}\]
 
\begin{lem}
\label{lem:ExistVect}Il existe un unique vecteur $\mathcal{V}_{\theta}\in V(M,\mathbb{Q}_{p^{s}})$
tel que \[
\mathcal{G}_{\theta}=\mathcal{F}-\mathcal{V}_{\theta}\quad\mbox{sur}\quad\mathcal{I}_{M}(G,\mathbb{Q}_{p^{s}}).\]
\end{lem}
\begin{proof}
On forme le diagramme commutatif suivant, analogue de \ref{eq:EquivIota}:
\[
\begin{array}{ccccc}
\mathcal{I}(J,\mathbb{Q}_{p^{s}}) & \stackrel{\iota}{\longrightarrow} & \mathcal{I}(M,\mathbb{Q}_{p^{s}}) & \stackrel{\theta}{\longrightarrow} & \mathcal{I}_{M}(G,\mathbb{Q}_{p^{s}})\\
\sigma\downarrow &  & \sigma\downarrow &  & \sigma\downarrow\\
\mathcal{I}(\sigma J,\mathbb{Q}_{p^{s}}) & \stackrel{\sigma\iota}{\longrightarrow} & \mathcal{I}(\sigma M,\mathbb{Q}_{p^{s}}) & \stackrel{\sigma\theta}{\longrightarrow} & \mathcal{I}_{\sigma M}(G,\mathbb{Q}_{p^{s}})\\
\parallel &  & *\downarrow &  & *\downarrow\\
\mathcal{I}(J,\mathbb{Q}_{p^{s}}) & \stackrel{\iota}{\longrightarrow} & \mathcal{I}(M,\mathbb{Q}_{p^{s}}) & \stackrel{\theta'}{\longrightarrow} & \mathcal{I}_{M}(G,\mathbb{Q}_{p^{s}})\end{array}\]
Les flches verticales du haut proviennent du changement de base $\sigma:\mathbb{Q}_{p^{s}}\rightarrow\mathbb{Q}_{p^{s}}$,
tandis que les flches verticales marques du bas sont induites par
l'automorphisme intrieur $\Int(b)$ de $G$ (qui envoie $\sigma M$
sur $M$). La commutativit du premier carr de gauche est une tautologie,
celle du second rsulte de la formule $\Int(b)\circ\sigma(\iota)=\iota$.
La commutativit des deux carrs de droite dfinit les isomtries
$\sigma\theta$ et $\theta'$, qui sont donc ncessairement $(\sigma M)(\mathbb{Q}_{p^{s}})$
et $M(\mathbb{Q}_{p^{s}})$-quivariantes (respectivement), ainsi
que torales, selon la terminologie de~\cite[1.3.3]{La00}.
En particulier, $\theta'$ et $\theta$ ne diffrent que d'une translation
par un lment de $V'(M,\mathbb{Q}_{p^{s}})$ -- cf. \cite[2.1.5.ii]{La00}.
D'autre part, le compos des deux flches de droite est induit par
l'action $\mathcal{G}$ de l'lment $(b,\sigma)$ de $G(L)\rtimes\left\langle \sigma\right\rangle $
sur $\mathcal{I}(G,L)$, o l'on fait maintenant agir $G(L)$ sur
$\mathcal{I}(G,L)$ par automorphismes intrieurs, comme dans la section~\ref{sub:FonctAut}.
On a donc d'une part $\mathcal{G}_{\theta}=\mathcal{G}-\mathcal{V}'_{\theta}$
sur $\mathcal{I}_{M}(G,\mathbb{Q}_{p^{s}})$ pour un lment $\mathcal{V}'_{\theta}\in V'(M,\mathbb{Q}_{p^{s}})$,
et $\mathcal{F}=\mathcal{G}+v_{G}(b)$ sur tout $\mathcal{I}(G,L)$
avec $v_{G}(b)\in V(G,\mathbb{Q}_{p^{s}})$. Donc $\mathcal{G}_{\theta}=\mathcal{F}-\mathcal{V}_{\theta}$
sur $\mathcal{I}_{M}(G,\mathbb{Q}_{p^{s}})$, avec \[
\mathcal{V}_{\theta}=(\mathcal{V}'_{\theta},v_{G}(b))\quad\mbox{dans}\quad V(M,\mathbb{Q}_{p^{s}})=V'(M,\mathbb{Q}_{p^{s}})\times V(G,\mathbb{Q}_{p^{s}}).\]
L'unicit est vidente.
\end{proof}
Par dfinition de $\mathcal{G}_{\theta}$ et $\mathcal{V}_{\theta}$
et en utilisant \ref{sub:FonctGal}, on obtient une isomtrie \[
\mathcal{I}(J,\mathbb{Q}_{p})\stackrel{\theta\iota}{\longrightarrow}\left\{ x\in\mathcal{I}_{M}(G,\mathbb{Q}_{p^{s}}):\mathcal{G}_{\theta}x=x\right\} =\left\{ x\in\mathcal{I}_{M}(G,\mathbb{Q}_{p^{s}}):\mathcal{F}x=x+\mathcal{V}_{\theta}\right\} \]
qui est quivariante pour les actions de $J(\mathbb{Q}_{p})=\left\{ g\in M(\mathbb{Q}_{p^{s}}):b\sigma(g)=gb\right\} $. 
\begin{lem}
\label{lem:DefTrans}Soit $\mathcal{G}=\Int(b)\circ\sigma$ agissant
sur $V(M,\mathbb{Q}_{p^{s}})$ et \[
\mathcal{Z}_{\theta}={\textstyle \sum_{i=0}^{s-1}}{\textstyle \frac{s-1-2i}{2s}}\cdot\mathcal{G}^{i}(\mathcal{V}_{\theta})\in V(M,\mathbb{Q}_{p^{s}}).\]
Alors la translation par $\mathcal{Z}_{\theta}$ induit une isomtrie
$J(\mathbb{Q}_{p})$-quivariante\[
\left\{ x\in\mathcal{I}_{M}(G,\mathbb{Q}_{p^{s}}):\mathcal{F}x=x+\mathcal{V}_{\theta}\right\} \stackrel{+\mathcal{Z}_{\theta}}{\longrightarrow}\left\{ x\in\mathcal{I}_{M}(G,\mathbb{Q}_{p^{s}}):\mathcal{F}x=x+\nu\right\} .\]
\end{lem}
\begin{proof}
Puisque $\mathcal{G}_{\theta}$ est conjugu  $\sigma$, $\mathcal{G}_{\theta}^{s}=\mathrm{Id}$
sur $\mathcal{I}_{M}(G,\mathbb{Q}_{p^{s}})$. Or \[
\mathcal{G}_{\theta}^{s}=\left(\mathcal{F}-\mathcal{V}_{\theta}\right)^{s}=\mathcal{F}^{s}-{\textstyle \sum_{i=0}^{s-1}\mathcal{G}^{i}(\mathcal{V}_{\theta})}=\mathrm{Id}+s\nu-{\textstyle \sum_{i=0}^{s-1}\mathcal{G}^{i}(\mathcal{V}_{\theta})}\]
car $\mathcal{F}\circ(\mathrm{Id}-\mathcal{V}_{\theta})=(\mathrm{Id}-\mathcal{G}(\mathcal{V}_{\theta}))\circ\mathcal{F}$
et $\mathcal{F}^{s}=\mathrm{Id}+s\nu$ sur $\mathcal{I}_{M}(G,\mathbb{Q}_{p^{s}})$
d'aprs le lemme~\ref{lem:CalcMin}. On obtient donc l'lgante formule
\[
\nu={\textstyle \frac{1}{s}}{\textstyle \sum_{i=0}^{s-1}\mathcal{G}^{i}(\mathcal{V}_{\theta})}\quad\mbox{dans}\quad V(M,\mathbb{Q}_{p^{s}}).\]
On en dduit que $\mathcal{G}(\mathcal{Z}_{\theta})-\mathcal{Z}_{\theta}=\nu-\mathcal{V}_{\theta}$;
le lemme en rsulte immdiatement. 
\end{proof}
En combinant les lemmes~\ref{lem:CalcMin} et \ref{lem:DefTrans},
on obtient une isomtrie $J(\mathbb{Q}_{p})$-quivariante:\[
\iota_{\theta}=\left(\theta\circ\iota+\mathcal{Z}_{\theta}\right):\mathcal{I}(J,\mathbb{Q}_{p})\rightarrow\Min(\mathcal{F})\]
qui dpend  priori de $\theta$. Mais si l'on change $\theta$ en
$\theta'=\theta+v$ avec $v\in V(M,\mathbb{Q}_{p^{s}})$, alors $\mathcal{G}_{\theta'}=\mathcal{G}_{\theta}+v-\mathcal{G}(v)$,
donc $\mathcal{V}_{\theta'}=\mathcal{V}_{\theta}+\mathcal{G}(v)-v$
puis \[
\mathcal{Z}_{\theta'}=\mathcal{Z}_{\theta}+v_{0}-v\quad\mbox{et}\quad\iota_{\theta'}=\iota_{\theta}+v_{0}\quad\mbox{avec}\quad v_{0}={\textstyle \frac{1}{s}\sum_{i=0}^{s-1}\mathcal{G}^{i}v.}\]
Puisque $\iota:J\rightarrow M_{\mathbb{Q}_{p^{s}}}$ induit un isomorphisme
de $V(J,\mathbb{Q}_{p})$ sur $V(M,\mathbb{Q}_{p^{s}})^{\mathcal{G}=\mathrm{Id}}$,
il existe un unique $u_{0}\in V(J,\mathbb{Q}_{p})$ tel que $\iota(u_{0})=v_{0}$,
et alors\[
\forall x\in\mathcal{I}(J,\mathbb{Q}_{p}):\qquad\iota_{\theta'}(x)=\iota_{\theta}(x)+v_{0}=\iota_{\theta}(x+u_{0})\quad\mbox{dans}\quad\Min(\mathcal{F}).\]
Cette isomtrie est donc bien dfinie  translation prs par un lment
de $V(J,\mathbb{Q}_{p})$, ce qui achve la preuve du thorme.\end{proof}





\subsection{Estimation immobilire}

Pour tout $b\in G(L)$, notons $\mathcal{G}_{b}$ et $\mathcal{V}_{b}$
les isomtries de $\mathcal{I}'(G,L)$ et $V(G,L)$ induites par l'lment
$(b,\sigma)$ de $G(L)\rtimes\left\langle \sigma\right\rangle $,
de sorte que \[
\mathcal{F}_{b}=(\mathcal{G}_{b},\mathcal{V}_{b})\quad\mbox{sur}\quad\mathcal{I}(G,L)=\mathcal{I}'(G,L)\times V(G,L).\]
D'aprs \cite[Th. 4.1]{Pa99} il existe une constante $\alpha_{0}\in]0,\pi]$
indpendante de $b$ telle que\[
\forall x\in\mathcal{I}'(G,L):\qquad\mathrm{dist}\left(x,\mathcal{G}_{b}(x)\right)\geq\sin\left({\textstyle \frac{\alpha_{0}}{2}}\right)\cdot\mathrm{dist}\left(x,\mathrm{Min}\left(\mathcal{G}_{b}\right)\right).\]
D'autre part, $\mathcal{V}_{b}(x)=\sigma(x)+v_{G}(x)$ pour tout $x\in V(G,L)$.
Notant $r>0$ l'ordre de $\sigma$ sur $V(G,L)$, un calcul lmentaire
fournit la minoration suivante: \[
\forall x\in V(G,L):\qquad\mathrm{dist}\left(x,\mathcal{V}_{b}(x)\right)\geq2\sin\left({\textstyle \frac{\pi}{r}}\right)\cdot\mathrm{dist}\left(x,\mathrm{Min}\left(\mathcal{V}_{b}\right)\right).\]
On en dduit l'existence d'une constante $\kappa>0$ indpendante
de $b$ telle que \[
\forall x\in\mathcal{I}(G,L):\qquad\mathrm{dist}\left(x,\mathcal{F}_{b}(x)\right)\geq\max\left(\min\left(\mathcal{F}_{b}\right),\kappa\cdot\mathrm{dist}\left(x,\mathrm{Min}\left(\mathcal{F}_{b}\right)\right)\right).\]
Explicitement, $\kappa=\sin\left({\textstyle \frac{\alpha_{0}}{2}}\right)$
si $r=1$ et $\kappa=\min\left(\sin\left({\textstyle \frac{\alpha_{0}}{2}}\right),2\sin\left({\textstyle \frac{\pi}{r}}\right)\right)$
si $r>1$. Cette minoration amliore le rsultat principal de \cite{RZ} selon lequel il existe une constante $k>0$ telle que $\mathrm{dist}(x,\mathcal{F}_b (x)) > k \cdot \mathrm{dist} (x, \mathcal{I}(J_b, \Q_{p^s}))$ pour tout $x  \in\mathcal{I}(G,L)$.

\section{Exemple: $\GL_n$} \label{Exemple}

Pour les groupes classiques, on identifie souvent les immeubles de
Bruhat-Tits avec certains espaces de normes $p$-adiques. Nous illustrons
ici ce que donne la transposition de nos constructions cristallines
à ces modles plus concrets d'immeubles dans le cas du groupe linaire
gnral, c'est--dire pour les isocristaux sans structures additionnelles.

Soit donc $h$ un entier positif, $V$ un espace vectoriel de dimension
$h$ sur $\mathbb{Q}_{p}$ et $G=\mathrm{GL}(V)$. On fixe un lment
$b\in G(L)$ et on note $(N,\mathcal{F})$ l'isocristal dfini par
\[
N=V\otimes L\quad\mbox{et}\quad\mathcal{F}=b\circ(\mathrm{Id}_{V}\otimes\sigma).\]
On note $(N,\mathcal{F})=\oplus(N_{\lambda},\mathcal{F}_{\lambda})$
la dcomposition isocline de $(N,\mathcal{F})$ avec $\lambda\in\mathbb{Q}$,
$\lambda=\frac{d(\lambda)}{h(\lambda)}$ où $d(\lambda)$ et $h(\lambda)>0$
sont relativement premiers. On pose \[
\mathbb{F}(\lambda)=\left\{ x\in k:\sigma^{h(\lambda)}(x)=x\right\} \quad\mbox{et}\quad K(\lambda)=\left\{ x\in L:\sigma^{h(\lambda)}(x)=x\right\} \]
et on note $(N_{\lambda}^{0},\mathcal{F}_{\lambda}^{0})$ l'isocristal
sur $\mathbb{F}(\lambda)$ qui est donn par: \[
N_{\lambda}^{0}=\left\{ x\in N_{\lambda}:\mathcal{F}_{\lambda}^{h(\lambda)}x=p^{d(\lambda)}x\right\} \quad\mbox{et}\quad\mathcal{F}_{\lambda}^{0}=\mathcal{F}_{\lambda}\vert N_{\lambda}^{0}.\]
Soient $J=\prod J_{\lambda}$ et $J_{\lambda}$ les groupes d'automorphismes
de $(N,\mathcal{F})$ et $(N_{\lambda},\mathcal{F}_{\lambda})$. Alors\[
J_{\lambda}(\mathbb{Q}_{p})=\Aut_{L}\left(N_{\lambda},\mathcal{F}_{\lambda}\right)=\Aut_{K(\lambda)}\left(N_{\lambda}^{0},\mathcal{F}_{\lambda}^{0}\right)=\Aut_{\mathbb{D}(\lambda)}\left(N_{\lambda}^{0}\right)\]
où $\mathbb{D}(\lambda)$ est le corps gauche de centre $\mathbb{Q}_{p}$
engendr sur $K(\lambda)$ par un lment $\pi_{\lambda}$ tel que
$\pi_{\lambda}x=\sigma(x)\pi_{\lambda}$ pour tout $x\in K(\lambda)$
et $\pi_{\lambda}^{h(\lambda)}=p^{d(\lambda)}$, que l'on fait agir
sur le $K(\lambda)$-espace vectoriel $N_{\lambda}^{0}$ par $\pi_{\lambda}\mapsto\mathcal{F}_{\lambda}^{0}$. 

On munit $A\in\{L,K(\lambda),\mathbb{D}(\lambda)\}$ de la valuation
normalise par $\left|p\right|=1/p$; si $X$ est un $A$-espace vectoriel
de dimension finie, on note $\mathcal{N}(X,A)$ l'ensemble des $A$-normes
de $X$, c'est--dire l'ensemble des normes $\alpha:X\rightarrow\mathbb{R}_{+}$
telles que $\alpha(ax)=\left|a\right|\alpha(x)$ pour tout $a\in A$
et tout $x\in X$. L'immeuble $\mathcal{I}(G,L)$ s'identifie à l'ensemble\emph{
$\mathcal{N}(N,L)$ }des $L$-normes de $N$, où $G(L)\rtimes\left\langle \sigma\right\rangle $
agit par $(g\cdot\alpha)(x)=\alpha(g^{-1}x)$. L'immeuble $\mathcal{I}(M,L)\simeq\mathcal{I}_{M}(G,L)$
du Levi $M=\prod\mathrm{GL}(N_{\lambda})$ de $G_{L}$ correspond
au sous-ensemble $\prod\mathcal{N}(N_{\lambda},L)$ des $L$-normes
de $N$ qui se dcomposent selon $N=\oplus N_{\lambda}$. Enfin \[
\Min(\mathcal{F})={\textstyle \prod}\Min(\mathcal{F}_{\lambda})\subset{\textstyle \prod}\Min(\mathcal{F}_{\lambda}^{h(\lambda)})\subset{\textstyle \prod}\mathcal{N}(N_{\lambda},L)\subset\mathcal{N}(N,L)\]
où les sous-ensembles $\Min(\mathcal{F}_{\lambda})$ et $\Min(\mathcal{F}_{\lambda}^{h(\lambda)})$
de $\mathcal{N}(N_{\lambda},L)$ sont dfinis par\[
\begin{array}{lcl}
\Min(\mathcal{F}_{\lambda}) & = & \left\{ \alpha\in\mathcal{N}(N_{\lambda},L):\forall x\in N_{\lambda},\,(\mathcal{F}_{\lambda}\cdot\alpha)(x)=p^{\lambda}\cdot\alpha(x)\right\} ,\\
\Min(\mathcal{F}_{\lambda}^{h(\lambda)}) & = & \left\{ \alpha\in\mathcal{N}(N_{\lambda},L):\forall x\in N_{\lambda},\,(\mathcal{F}_{\lambda}^{h(\lambda)}\cdot\alpha)(x)=p^{d(\lambda)}\cdot\alpha(x)\right\} \\
 & = & \left\{ \alpha\in\mathcal{N}(N_{\lambda},L):\forall x\in N_{\lambda},\,\alpha(\mathcal{F}_{\lambda}^{h(\lambda)}x)=\alpha(p^{d(\lambda)}x)\right\} .\end{array}\]
La dernire galit identifie $\Min(\mathcal{F}_{\lambda}^{h(\lambda)})$
et $\mathcal{N}(N_{\lambda}^{0},K(\lambda))$, puis \[
\Min(\mathcal{F}_{\lambda})=\left\{ \alpha\in\mathcal{N}(N_{\lambda}^{0},K(\lambda)):\forall x\in N_{\lambda}^{0},\,\alpha(\pi_{\lambda}x)=\left|\pi_{\lambda}\right|\alpha(x)\right\} =\mathcal{N}(N_{\lambda}^{0},\mathbb{D}(\lambda)).\]
Puisque $\mathcal{I}(J_{\lambda},\mathbb{Q}_{p})\simeq\mathcal{N}(N_{\lambda}^{0},\mathbb{D}(\lambda))$,
on a ainsi vrifi que $\Min(\mathcal{F})\simeq\mathcal{I}(J,\mathbb{Q}_{p})$. 

\begin{rmk}
Un cristal dans $N$ est un $W$-rseau $M$ de $N$ qui est stable
par $\mathcal{F}$ et $\mathcal{V}=p\mathcal{F}^{-1}$. De tels rseaux
n'existent que lorsque toutes les pentes $\lambda$ de $N$ sont comprises
entre $0$ et $1$. Sous cette hypothse, la proposition \cite[Prop. 5.17]{LNV} montre que les cristaux minimaux au sens de Oort sont les
boules des $L$-normes de $N$ qui sont dans $\Min(\mathcal{F})=\prod\Min(\mathcal{F}_{\lambda})$: pour tout $\alpha \in \Min(\mathcal{F})$ et $r>0$, la boule $B(\alpha \leq r) := \left\{ x | \alpha(x) \leq r \right\}$ est un cristal minimal, et tous les cristaux minimaux sont de cette forme. Pour $\GL_n$, l'unicit bien connue du cristal minimal  isomorphisme prs dcoule du fait que tous les sommets de $\mathcal{I}(J,\Q_p)$ sont conjugus sous $J(\Q_p)$ quand $J$ est une forme intrieure d'un Lvi de $\GL_n$. Pour un groupe $G$ arbitraire, le nombre d'orbites de cette action est seulement fini.

\end{rmk}

{\bf Remerciements:}

\noindent
Nous remercions l'Institut Max Planck de Bonn pour ses conditions de travail et son hospitalit idales, pour toute la dure de l'anne 2011 pour le second auteur, et pour un sjour d'une semaine en aot 2011 pour le premier auteur. Ce travail a aussi bnfici du soutien de l'ANR/ArShiFo, 
et de nombreux changes avec Guy Rousseau.

\bibliographystyle{plain}
\bibliography{MarcHub.bib}

\begin{thebibliography}{10}

\bibitem{BrHa99}
M.~R. Bridson and A.~Haefliger.
\newblock {\em Metric spaces of non-positive curvature}, volume 319 of {\em
  Grundlehren der Mathematischen Wissenschaften [Fundamental Principles of
  Mathematical Sciences]}.
\newblock Springer-Verlag, Berlin, 1999.

\bibitem{BrTi72}
F.~Bruhat and J.~Tits.
\newblock Groupes r{\'e}ductifs sur un corps local.
\newblock {\em Inst. Hautes {\'E}tudes Sci. Publ. Math.}, (41):5--251, 1972.

\bibitem{Ko84}
R.~E. Kottwitz.
\newblock Shimura varieties and twisted orbital integrals.
\newblock {\em Math. Ann.}, 269(3):287--300, 1984.

\bibitem{Ko85}
R.~E. Kottwitz.
\newblock Isocrystals with additional structure.
\newblock {\em Compositio Math.}, 56(2):201--220, 1985.

\bibitem{Ko97}
R.~E. Kottwitz.
\newblock Isocrystals with additional structure. {II}.
\newblock {\em Compositio Math.}, 109(3):255--339, 1997.

\bibitem{La00}
E.~Landvogt.
\newblock Some functorial properties of the {B}ruhat-{T}its building.
\newblock {\em J. Reine Angew. Math.}, 518:213--241, 2000.

\bibitem{LNV}
E.~Lau, M.-H. Nicole, and A.~Vasiu.
\newblock Stratifications of {N}ewton polygon strata and {T}raverso's
  conjectures for $p$-divisible groups.
\newblock {\em Pr{\'e}publication 2009}.

\bibitem{Pa99}
A.~Parreau.
\newblock Immeubles affines : construction par les normes et {\'e}tude des
  isom{\'e}tries.
\newblock In {\em In Crystallographic groups and their generalizations
  (Kortrijk, 1999)}, volume 262 of {\em Contemp. Math.}, pages 263--302. Amer.
  Math. Soc., Providence, RI, 2000.

\bibitem{PrYu02}
G.~Prasad and J.-K. Yu.
\newblock On finite group actions on reductive groups and buildings.
\newblock {\em Invent. Math.}, 147(3):545--560, 2002.

\bibitem{RaRi96}
M.~Rapoport and M.~Richartz.
\newblock On the classification and specialization of {$F$}-isocrystals with
  additional structure.
\newblock {\em Compositio Math.}, 103(2):153--181, 1996.

\bibitem{RaZi96}
M.~Rapoport and Th. Zink.
\newblock {\em Period spaces for {$p$}-divisible groups}, volume 141 of {\em
  Annals of Mathematics Studies}.
\newblock Princeton University Press, Princeton, NJ, 1996.

\bibitem{RZ}
M.~Rapoport and Th. Zink.
\newblock A finiteness theorem in the {B}ruhat-{T}its building: an application
  of {L}andvogt's embedding theorem.
\newblock {\em Indag. Mathem., N.S.}, 10(3):449--458, 1999.

\bibitem{ReThWe10}
B.~R{\'e}my, A.~Thuillier, and A.~Werner.
\newblock Bruhat-{T}its theory from {B}erkovich's point of view. {I}.
  {R}ealizations and compactifications of buildings.
\newblock {\em Ann. Sci. {\'E}c. Norm. Sup{\'e}r. (4)}, 43(3):461--554, 2010.

\bibitem{Ro77}
G.~Rousseau.
\newblock {\em Immeubles des groupes r{\'e}ductifs sur les corps locaux}.
\newblock U.E.R. Math{\'e}matique, Universit{\'e} Paris XI, Orsay, 1977.
\newblock Th{\`e}se de doctorat, Publications Math{\'e}matiques d'Orsay, No.
  221-77.68.

\bibitem{Ro01}
G.~Rousseau.
\newblock Exercices m{\'e}triques immobiliers.
\newblock {\em Indag. Math. (N.S.)}, 12(3):383--405, 2001.

\bibitem{Ti79}
J.~Tits.
\newblock Reductive groups over local fields.
\newblock In {\em Automorphic forms, representations and {$L$}-functions
  ({P}roc. {S}ympos. {P}ure {M}ath., {O}regon {S}tate {U}niv., {C}orvallis,
  {O}re., 1977), {P}art 1}, Proc. Sympos. Pure Math., XXXIII, pages 29--69.
  Amer. Math. Soc., Providence, R.I., 1979.

\bibitem{VoWe11}
I.~Vollaard and T.~Wedhorn.
\newblock The supersingular locus of the {S}himura variety of {GU}$(1,n-1)$
  {II}.
\newblock {\em Inventiones Math.}, (184):591--627, 2011.

\end{thebibliography}



\end{document}